\documentclass[11pt]{article}
\usepackage{arxiv}
\usepackage{amsfonts}
\usepackage{amsmath}
\usepackage{amsthm}
\usepackage{pxfonts}
\usepackage{mathtools}
\usepackage{tocloft}
\usepackage{todonotes}
\usepackage{extarrows}
\usepackage{titlesec}
\usepackage{enumitem}
\usepackage[hidelinks]{hyperref}
\usepackage{subcaption}

\theoremstyle{definition}
\newtheorem{definition}{Definition}[]

\newtheorem{remark}[definition]{Remark}

\theoremstyle{plain}
\newtheorem{theorem}[definition]{Theorem}

\newtheorem{proposition}[definition]{Proposition}
\newtheorem{lemma}[definition]{Lemma}
\newtheorem{setting}[definition]{Setting}

\newcommand{\norm}[1]{\left\lVert#1\right\rVert}

\definecolor{dkgold}{rgb}{0.639, 0.53, 0.19}

\title{Uniform Convergence Guarantees for the Deep Ritz Method for Nonlinear Problems}
\date{\today}
\author{
    Patrick Dondl \\
    Department of Applied Mathematics\\
    University of Freiburg \\
    Hermann-Herder-Stra\ss e 10, 79104 Freiburg i. Br., Germany \\
    \texttt{patrick.dondl@mathematik.uni-freiburg.de}
   \And
    Johannes M\"uller \\
    Max Planck Institute for Mathematics in the Sciences,\\
    Inselstra\ss e 22, 04103 Leipzig, Germany\\
    \texttt{jmueller@mis.mpg.de}
   \And
    Marius Zeinhofer \\
    Department of Applied Mathematics\\
    University of Freiburg\\
    Hermann-Herder-Stra\ss e 10, 79104 Freiburg i. Br., Germany \\
    \texttt{marius.zeinhofer@mathematik.uni-freiburg.de}
}

\begin{document}

\maketitle

\begin{abstract}
    We provide convergence guarantees for the Deep Ritz Method for abstract variational energies. Our results cover non-linear variational problems such as the $p$-Laplace equation or the Modica-Mortola energy with essential or natural boundary conditions. Under additional assumptions, we show that the convergence is uniform across %
    bounded families of right-hand sides.
\end{abstract}

\section{Introduction}

The idea of the Deep Ritz Method is to use variational energies as an objective function for neural network training %
to obtain a finite dimensional optimization problem that allows to solve the underlying partial differential equation approximately. The idea of deriving a finite dimensional optimization problem from variational energies dates back to \cite{ritz1909neue}, was widely popularised in the context of finite element methods \cite[see, e.g.,][]{braess2007finite} and was recently revived by \cite{weinan2018deep} using deep neural networks. In the following, we give a more thorough introduction to the Deep Ritz Method. Let $\Omega\subset \mathbb R^d$ be a bounded domain and consider the variational energy corresponding to the Lagrangian $L$ and a force $f$
\begin{equation}\label{lagrangian}
    E\colon X\to\mathbb R, \quad E(u) = \int_\Omega L(\nabla u(x), u(x), x) - f(x)u(x)\mathrm dx,
\end{equation}
defined on a suitable function space $X$, usually a Sobolev space $W^{1,p}(\Omega)$. One is typically interested in minimizers of $E$ on subsets $U\subset X$ where $U$ encodes further physical constraints, such as boundary conditions. Here, we consider either unconstrained problems or zero Dirichlet boundary conditions and use the notation $U=X_0$ for the latter case. In other words, for zero boundary conditions, one aims to find
\begin{equation}\label{eq:continuous_optimization_problem}
    u \in \underset{v\in X_0}{\operatorname{argmin}}\int_\Omega L(\nabla v(x),v(x),x) - f(x)v(x)\mathrm dx.
\end{equation}
To solve such a minimization problem numerically, the idea dating back to \cite{ritz1909neue} is to use a parametric ansatz class
\begin{equation}\label{ansatz_space}
    A \coloneqq \{ u_\theta \in X \mid \theta \in \Theta \subset \mathbb{R}^P \} \subset U
\end{equation}
and to consider the finite dimensional minimization problem of finding
\begin{equation*}
    \theta^* \in \underset{\theta \in \Theta}{\operatorname{argmin}}\int_\Omega L(\nabla v_\theta(x),v_\theta(x),x) - f(x)v_\theta(x)\mathrm dx
\end{equation*}
which can be approached by different strategies, depending on the class $A$. For instance, if $A$ is chosen to be a finite element ansatz space or polynomials and the structure of $E$ is simple enough, one uses optimality conditions to solve this problem.

In this manuscript, we focus on ansatz classes that are given through (deep) neural networks. When choosing such ansatz functions, the method is known as the Deep Ritz Method and was recently proposed by \cite{weinan2018deep}. Neural network type ansatz functions possess a parametric form as in \eqref{ansatz_space}, however, it is difficult to impose zero boundary conditions on the ansatz class $A$. To circumvent this problem, one can use a penalty approach, relaxing the energy to the full space, but penalizing the violation of zero boundary conditions, to include these. This means, for a penalization parameter $\lambda > 0$ one aims to find
\begin{equation}\label{eq:deep_ritz_boundary_penalty}
    \theta_\lambda^* \in \underset{\theta \in \Theta}{\operatorname{argmin}}\int_\Omega L(\nabla v_\theta(x),v_\theta(x),x) - f(x)v_\theta(x)\mathrm dx + \lambda \int_{\partial\Omega}u^2_\theta\mathrm ds.
\end{equation}

The idea of using neural networks for the approximate solution of PDEs can be traced back at least to the works of~\cite{lee1990neural, dissanayake1994neural, takeuchi1994neural, lagaris1998artificial}. Since the recent successful application of neural network based methods to stationary and instationary PDEs by~\cite{weinan2017deep, weinan2018deep, sirignano2018dgm}, there is an ever growing body of theoretical works contributing to the understanding of these approaches. For a collection of the different methods we refer to the overview articles by~\cite{beck2020overview, han2020algorithms}.

The error in the deep Ritz method, which decomposes into an approximation, optimization and generalization term, has been studied by~\cite{luo2020two, xu2020finite, duan2021convergence, hong2021priori, jiao2021error, lu2021priori, lu2021machine, muller2021error}. However, those works either consider non essential boundary conditions or  they require a term with a positive potential, apart from~\cite{muller2021error}. This excludes the prototypical Poisson equation, which was originally treated by the deep Ritz method by~\cite{weinan2018deep}. More importantly, those works only study linear problems, which excludes many important applications.%

In this work, we thus study the convergence of the Deep Ritz Method when a sequence of growing ansatz classes $(A_n)_{n\in \mathbb{N}}$, given through parameter sets $\Theta_n$ and a penalization of growing strength $(\lambda_n)_{n\in\mathbb N}$ with $\lambda_n \nearrow \infty$, is used in the optimization problem \eqref{eq:deep_ritz_boundary_penalty} with more modest assumptions on $L$, $f$, and $\Omega$.

Denote a sequence of (almost) minimizing parameters of problem \eqref{eq:deep_ritz_boundary_penalty} with parameter set $\Theta_n$ and penalization $\lambda_n$ by $\theta_n$.
We then see that under mild assumptions on $(A_n)_{n\in\mathbb{N}}$ and $E$, the sequence $(u_{\theta_n})_{n\in\mathbb N}$ of (almost) minimizers converges weakly in $X$ to the solution of the continuous problem, see Theorem~\ref{thm:abstract_gamma_convergence_thm} in Section~\ref{section:plain_gamma_convergence}. We then strengthen this result in Section~\ref{section:uniform_gamma_covergence} where we show that the aforementioned convergence is uniform across certain bounded families of right-hand sides $f$, see Theorem~\ref{absunires}. This means that a fixed number of degrees of freedom in the ansatz class can be used independently of the right hand side to achieve a given accuracy. Alternatively, given a discretization of the space of right hand sides, one may discretize the solution operator that maps $f$ to the minimizer $u$ of \eqref{eq:continuous_optimization_problem} and still obtain a convergence guarantee (although this is not necessarily a viable numerical approach).

To the best of our knowledge, our results currently comprise the only convergence guarantees for the Deep Ritz Method for non-linear problems.
However, since we prove these results using $\Gamma$-convergence methods, no rates of convergence are be obtained -- as mentioned above, for linear elliptic equations some error decay estimates are known.
Our results also do not provide insight into the finite dimensional optimization problem \eqref{eq:deep_ritz_boundary_penalty}
which is a challenging problem in its own right, see for instance \cite{wang2021understanding, courte2021robin}. However, they guarantee that given one is able to solve \eqref{eq:deep_ritz_boundary_penalty} to a reasonable accuracy, one is approaching the solution of the continuous problem \eqref{eq:continuous_optimization_problem}.

Our results are formulated for neural network type ansatz functions due to the current interest in using these in numerical simulations, yet other choices are possible. For instance, our results do apply directly to finite element functions.

The remainder of this work is organized as follows. Section \ref{sec:preliminaries} discusses some preliminaries and the used notation. The main results, namely $\Gamma$-convergence and uniformity of convergence are provided in Sections \ref{section:plain_gamma_convergence} and \ref{section:uniform_gamma_covergence}, respectively. Finally, in Section~\ref{section:examples} we discuss how the $p$-Laplace and a phase field model fit into our general framework.

\section{Notation and preliminaries} \label{sec:preliminaries}

We fix our notation and present the tools that our analysis relies on.

\subsection{Notation of Sobolev spaces and Friedrich's inequality}

We denote the space of functions on \(\Omega\subseteq\mathbb R^d\) that are integrable in $p$-th power by \(L^p(\Omega)\), where we assume that $p\in[1,\infty)$. Endowed with
\[\lVert u\rVert_{L^p(\Omega)}^p \coloneqq \int_\Omega \lvert u\rvert^p\mathrm dx \]
this is a Banach space, i.e., a complete normed space. If \(u\) is a multivariate function with values in \(\mathbb R^m\) we interpret \(\lvert\cdot\rvert\) as the Euclidean norm.
We denote the subspace of \(L^p(\Omega)\) of functions with weak derivatives up to order \(k\) in $L^p(\Omega)$ by \(W^{k,p}(\Omega)\), which is a Banach space with the norm
\[ \lVert u\rVert_{W^{k,p}(\Omega)}^p \coloneqq \sum_{l=0}^k\lVert D^{l} u \rVert_{L^p(\Omega)}^p. \]
This space is called a \emph{Sobolev space} and we denote its dual space, i.e., the space consisting of all bounded and linear functionals on \(W^{k,p}(\Omega)\) by \(W^{k,p}(\Omega)^\ast\). The closure of all compactly supported smooth functions \(\mathcal C_c^\infty(\Omega)\) in \(W^{k,p}(\Omega)\) is denoted by \(W^{k,p}_0(\Omega)\). %
It is well known that if $\Omega$ has a Lipschitz continuous boundary the operator that restricts a Lipschitz continuous function on \(\overline{\Omega}\) to the boundary admits a linear and bounded extension $\operatorname{tr}\colon W^{1,p}(\Omega)\to L^p(\partial\Omega)$. %
This operator is called the \emph{trace operator} and its kernel is precisely \(W^{1,p}_0(\Omega)\). Further, we write \(\lVert u\rVert_{L^p(\partial\Omega)}\) whenever we mean \(\lVert \operatorname{tr}(u)\rVert_{L^p(\partial\Omega)}\). In the following we mostly work with the case $p=2$ and write $H^k_{(0)}(\Omega)$ instead of $W^{k,2}_{(0)}(\Omega)$. %

In order to study the boundary penalty method we use the Friedrich inequality which states that the \(L^p(\Omega)\) norm of a function can be estimated by the norm of its gradient and boundary values. We refer to \cite{graser2015note} for a proof.  %

\begin{proposition}[Friedrich's inequality]\label{poinclemma}
    Let \(\Omega\subseteq\mathbb R^d\) be a bounded and open set with Lipschitz boundary \(\partial\Omega\) and $p\in(1,\infty)$. Then there exists a constant \(c>0\) such that
    \begin{equation}\label{Poincarefinal}
    \lVert u\rVert_{W^{1,p}(\Omega)}^p\le c^p \cdot \left(\lVert \nabla u\rVert_{L^p(\Omega)}^p + \lVert u\rVert_{L^p(\partial\Omega)}^p\right) \quad \text{for all } u\in W^{1,p}(\Omega).
    \end{equation}
\end{proposition}

\subsection{Neural Networks}\label{sec:neural_networks}
Here we introduce our notation for the functions represented by a feedforward neural network. Consider natural numbers \(d, m, L, N_0, \dots, N_L\in\mathbb N\) and let
\[\theta = \left((A_1, b_1), \dots, (A_L, b_L)\right)\]
be a tupel of matrix-vector pairs where \(A_l\in\mathbb R^{N_{l}\times N_{l-1}}, b_l\in\mathbb R^{N_l}\) and \(N_0 = d, N_L = m\). Every matrix vector pair \((A_l, b_l)\) induces an affine linear map \(T_l\colon \mathbb R^{N_{l-1}} \to\mathbb R^{N_l}\). The \emph{neural network function with parameters} \(\theta\) and with respect to some \emph{activation function} \(\rho\colon\mathbb R\to\mathbb R\) is the function
\[u^\rho_\theta\colon\mathbb R^d\to\mathbb R^m, \quad x\mapsto T_L(\rho(T_{L-1}(\rho(\cdots \rho(T_1(x)))))).\]
The set of all neural network functions of a certain architecture is given by $\{ u^\rho_\theta\mid\theta\in\Theta\}$, where $\Theta$ collects all parameters of the above form with respect to fixed natural numbers $d,m,L,N_0,\dots,N_L$. If we have \(f=u_\theta^\rho\) for some \(\theta\in\Theta\) we say the function \(f\) can be \emph{realized} by the neural network $\mathcal F^\rho_{\Theta}$. Note that we often drop the superscript $\rho$ if it is clear from the context.

A particular activation function often used in practice and relevant for our results is the \emph{rectified linear unit} or \emph{ReLU activation function}, which is defined via \(x\mapsto \max\left\{ 0, x\right\}\).
\cite{arora2016understanding} showed that the class of ReLU networks coincides with the class of continuous and piecewise linear functions. In particular they are weakly differentiable. Since piecewise linear functions are dense in \(H^1_0(\Omega)\) we obtain the following universal approximation result which we prove in detail in the appendix.

\begin{theorem}[Universal approximation with zero boundary values]\label{UniversalApproximation}
Consider an open set \(\Omega\subseteq\mathbb R^d\) and fix a function \(u\in W^{1,p}_0(\Omega)\) with $p\in[1,\infty)$. Then for all \(\varepsilon>0\) there exists \(u_\varepsilon\in W^{1,p}_0(\Omega)\) that can be realized by a ReLU network of depth \(\lceil \log_2(d+1)\rceil +1\) such that \[\left\lVert u - u_\varepsilon \right\rVert_{W^{1,p}(\Omega)}\le \varepsilon.\]
\end{theorem}

To the best of our knowledge this is the only available universal approximation results where the approximating neural network functions are guaranteed to have zero boundary values. This relies on the special properties of the ReLU activation function and it is unclear for which classes of activation functions universal approximation with zero boundary values hold.

\subsection{Gamma Convergence}\label{section:gamma_convergence}
We recall the definition of $\Gamma$-convergence with respect to the weak topology of reflexive Banach spaces. For further reading we point the reader towards \cite{dal2012introduction}.
\begin{definition}[\(\Gamma\)-convergence]
    Let \(X\) be a reflexive Banach space as well as \(F_n, F\colon X\to(-\infty, \infty]\). Then \((F_n)_{n\in\mathbb N}\) is said to be \emph{\(\Gamma\)-convergent} to \(F\) if the following two properties are satisfied.
    \begin{enumerate}
        \item \emph{Liminf inequality:} For every \(x\in X\) and \((x_n)_{n\in\mathbb N}\) with \(x_n\rightharpoonup x\) we have
            \[F(x) \le \liminf_{n\to\infty} F_n(x_n).\]
        \item \emph{Recovery sequence:} For every \(x\in X\) there is \((x_n)_{n\in\mathbb N}\) with \(x_n\rightharpoonup x\) such that
            \[F(x) = \lim_{n\to\infty} F_n(x_n).\]
\end{enumerate}
    The sequence \((F_n)_{n\in\mathbb N}\) is called \emph{equi-coercive} if the set
        \[\bigcup_{n\in\mathbb N}\Big\{x\in X \mid F_n(x)\le r \Big\} \]
    is bounded in \(X\) (or equivalently relatively compact with respect to the weak topology) for all \(r\in\mathbb R\). We say that a sequence \((x_n)_{n\in\mathbb N}\) are \emph{quasi minimizers} of the functionals \((F_n)_{n\in\mathbb N}\) if we have
    \[F_n(x_n) \le \inf_{x\in X} F_n(x) + \delta_n\]
    where \(\delta_n\to0\).
\end{definition}

We need the following property of \(\Gamma\)-convergent sequences. We want to emphasise the fact that there are no requirements regarding the continuity of any of the functionals and that the functionals \((F_n)_{n\in\mathbb N}\) are not assumed to admit minimizers.

\begin{theorem}[Convergence of quasi-minimizers]\label{thm:gist_of_gamma_convergence}
    Let \(X\) be a reflexive Banach space and \((F_n)_{n\in\mathbb N}\) be an equi-coercive sequence of functionals that \(\Gamma\)-converges to \(F\). Then, any sequence $(x_n)_{n\in\mathbb{N}}$ of quasi-minimizers of $(F_n)_{n\in\mathbb{N}}$ is relatively compact with respect to the weak topology of $X$ and every weak accumulation point of $(x_n)_{n\in\mathbb{N}}$ is a global minimizer of $F$. Consequently, if $F$ possesses a unique minimizer $x$, then $(x_n)_{n\in\mathbb{N}}$ converges weakly to $x$.
\end{theorem}

\section{\texorpdfstring{Abstract $\Gamma$-Convergence Result for the Deep Ritz Method}{Abstract Gamma-Convergence Result for the Deep Ritz Method}}\label{section:plain_gamma_convergence}
For the abstract results we work with an abstract energy $E\colon X\to \mathbb{R}$, instead of an integral functional of the form \eqref{lagrangian}. This reduces technicalities in the proofs and separates abstract functional analytic considerations from applications.
\begin{setting}\label{setting:gamma_convergence}
Let $(X,\norm{\cdot}_X)$ and $(B,\norm{\cdot}_B)$ be reflexive Banach spaces and $\gamma \in \mathcal{L}(X,B)$ be a continuous linear map. We set $X_0$ to be the kernel of $\gamma$, i.e.,\ $X_0 = \gamma^{-1}(\{0\})$. Let $\rho\colon\mathbb R \to \mathbb R$ be some activation function and denote by $(\Theta_n)_{n\in\mathbb N}$ a sequence of neural network parameters. We assume that any function represented by such a neural network is a member of $X$ and we define
\begin{equation*}
    A_n \coloneqq \left\{ x_\theta \mid \theta \in \Theta_n \right\} \subset X.
\end{equation*}
Here, $x_\theta$ denotes the function represented by the neural network with the parameters $\theta$. Let $E\colon X\to(-\infty,\infty]$ be a functional and $(\lambda_n)_{n\in \mathbb N}$ a sequence of real numbers with $\lambda_n\to\infty$. Furthermore, let $p\in (1,\infty)$ and $f\in X^*$ be fixed and define the functional $F^f_n\colon X\to(-\infty,\infty]$ by
\begin{align*}
        F^f_n(x) = \begin{cases}\; \displaystyle E(x) + \lambda_n\lVert\gamma(x)\rVert^p_B - f(x) \quad&\text{for }x\in A_n, \\[.4cm] \;\infty &\text{otherwise }, \end{cases}
    \end{align*}
    as well as $F^f\colon X\to(-\infty,\infty]$ by
    \begin{align*}
        F^f(x) = \begin{cases} \; \displaystyle E(x) - f(x) \quad&\text{for } x\in X_0, \\[.4cm] \; \infty   &\text{otherwise }.\end{cases}
    \end{align*}
    Then assume the following holds:
    \begin{enumerate}[label=\textup{(A\arabic*)}]
        \item\label{A1} For every $x\in X_0$ there is %
        $x_n\in A_n$ such that $x_n\to x$ and $\lambda_n\lVert \gamma(x_n) \rVert^p_B \to 0$ for $n\to \infty$.
        \item\label{A2} The functional $E$ is bounded from below, weakly lower semi-continuous with respect to the weak topology of $(X,\norm{\cdot}_X)$ and continuous with respect to the norm topology of $(X,\norm{\cdot}_X)$.
        \item\label{A3} The sequence $(F^f_n)_{n\in\mathbb N}$ is equi-coercive with respect to the norm $\norm{\cdot}_X$.
    \end{enumerate}
\end{setting}
\begin{remark} We discuss the Assumptions \ref{A1} to \ref{A3} in view of their applicability to concrete problems.
\begin{enumerate}
    \item In applications, $(X,\lVert\cdot\rVert_X)$ will usually be a Sobolev space with its natural norm, the space $B$ contains boundary values of functions in $X$ and the operator $\gamma$ is a boundary value operator, e.g. the trace map. However, if the energy $E$ is coercive on all of $X$, i.e. without adding boundary terms to it, we might choose $\gamma = 0$ and obtain $X_0 = X$. This is the case for non-essential boundary value problems.
    \item The Assumption \ref{A1} compensates that in general, we cannot penalize with arbitrary strength. However, if we can approximate any member of $X_0$ by a sequence $x_{\theta_n}\in A_n\cap X_0$ then any divergent sequence $(\lambda_n)_{n\in\mathbb{N}}$ can be chosen. This is for example the case for the ReLU activation function and the space $X_0=H^1_0(\Omega)$. More precisely, we can choose $A_n$ to be the class of functions expressed by a
   (fully connected) ReLU network of depth $\lceil \log_2(d+1)\rceil +1$ and width $n$, see Theorem~\ref{UniversalApproximation}.
\end{enumerate}
\end{remark}
\begin{theorem}[$\Gamma$-convergence]\label{thm:abstract_gamma_convergence_thm}
    Assume we are in Setting \ref{setting:gamma_convergence}. {Then the sequence $(F^f_n)_{n\in\mathbb{N}}$ of functionals $\Gamma$-converges towards $F^f$.} %
    In particular, if $(\delta_n)_{n\in\mathbb{N}}$ is a sequence of non-negative real numbers converging to zero, %
    any sequence of $\delta_n$-quasi minimizers of $F_n^f$ is bounded and all its weak accumulation points are minimizers of $F^f$.
    If additionally $F^f$ possesses a unique minimizer $x^f \in X_0$, any sequence of $\delta_n$-quasi minimizers converges to $x^f$ in the weak topology of $X$. %
\end{theorem}
\begin{proof}
We begin with the limes inferior inequality.
Let $x_n\rightharpoonup x$ in $X$ and assume that $x\notin X_0$. Then $f(x_n)$ converges to $f(x)$ as real numbers and $\gamma(x_n)$ converges weakly to $\gamma(x)\neq 0$ in $B$. Combining this with the weak lower semicontinuity of $\rVert\cdot\lVert^p_B $ we get, using the boundedness from below, that
\[
\liminf_{n\to\infty}F_n^{f}(x_n) \geq \inf_{x\in X} E(x) + \liminf_{n\to\infty}\lambda_n\lVert\gamma(x_n)\rVert^p_B - \lim_{n\to\infty}f(x_n) = \infty.
\]
Now let $x\in X_0$. Then by the weak lower semicontinuity of $E$ we find
\[
\liminf_{n\to\infty}F_n^{f}(x_n) \geq \liminf_{n\to\infty}E(x_n) - f(x) \geq E(x) -f(x) = F^f(x).
\]
Now let us have a look at the construction of the recovery sequence. For $x\notin X_0$ we can choose the constant sequence and estimate
\[
F^{f}_n(x_n) \geq E(x) + \lambda_n\lVert\gamma(x)\rVert_B^p - f(x).
\]
Hence we find that $F_n^{f_n}(x)\to \infty = F^f(x)$. If $x\in X_0$ we approximate it with a sequence $(x_n)\subseteq X$ according to Assumption \ref{A1}, such that $x_n\in A_n$ and $x_n\to x$ in $\lVert\cdot\rVert_X$ and $\lambda_n\lVert \gamma(x_n) \rVert_B^p \to 0$. It follows that
\[
F_n^{f}(x_n) = E(x_n) + \lambda_n\lVert x_n \rVert_B^p -f(x_n) \to E(x)-f(x) = F^f(x).
\]
\end{proof}
A sufficient criterion for equi-coercivity of the sequence $(F^f_n)_{n\in\mathbb{N}}$ from Assumption~\ref{A3} in terms of the functional $E$ is given by the following lemma. %
\begin{lemma}[Criterion for Equi-Coercivity]\label{lemma:coercivity_condition}
    Assume we are in Setting \ref{setting:gamma_convergence}. If there is a constant $c>0$ such that it holds for all $x\in X$ that
    \begin{equation*}
        E(x) + \lVert \gamma(x) \rVert_B^p \geq c \cdot \left( \norm{x}_X^p - \norm{x} - 1 \right),
    \end{equation*}
    then the sequence $(F^f_n)_{n\in\mathbb{N}}$ is equi-coercive.
\end{lemma}
\begin{proof}
    It suffices to show that the sequence
    \[
    G_n^f\colon X\to\mathbb{R}\quad\text{with}\quad G^f_n(x) = E(x) + \lambda_n\lVert\gamma(x)\rVert_B^p - f(x)
    \]
    is equi-coercive, as $G^f_n\leq F_n$. So let $r\in\mathbb{R}$ be given and assume that $r\geq G^f_n(x)$. We estimate assuming without loss of generality that $\lambda_n \geq 1$
    \begin{align*}
        r &\geq E(x) + \lambda_n\lVert\gamma(x)\rVert_B^p - f(x)
        \\
        &\geq c \cdot \big(\lVert x \rVert_X^p - \lVert x \rVert_X - 1\big) - \lVert f\rVert_{X^*}\cdot\lVert x\rVert_X
        \\
        &\geq \tilde c \cdot \big(\lVert x \rVert_X^p - \lVert x \rVert_X - 1\big).
    \end{align*}
    As $p> 1$, a scaled version of Young's inequality clearly implies a bound on the set
    \[
    \bigcup_{n\in\mathbb{N}}\big\{ x\in X\mid G_n(x) \leq r \big\}
    \]
    and hence the sequence $(F^f_n)_{n\in\mathbb{N}}$ is seen to be equi-coercive.
\end{proof}

\section{Abstract Uniform Convergence Result for the Deep Ritz Method}\label{section:uniform_gamma_covergence}
In this section we present an extension of Setting \ref{setting:gamma_convergence} that allows to prove uniform convergence results over certain bounded families of right-hand sides.
\begin{setting}\label{setting:uniform_gamma_convergence}
    Assume we are in Setting \ref{setting:gamma_convergence}. Furthermore, let there be an additional norm $\lvert\cdot\rvert$ on $X$ such that the dual space $(X,\lvert\cdot\rvert)^*$ is reflexive. However, we do not require $(X,\lvert\cdot\rvert)$ to be complete. Then, let the following assumptions hold
    \begin{enumerate}[label=\textup{(A\arabic*)}]
        \addtocounter{enumi}{3}
        \item\label{A4} The identity $\operatorname{Id}\colon (X,\norm{\cdot}_X)\to(X,\lvert\cdot\rvert)$ is completely continuous, i.e., maps weakly convergent sequences to strongly convergent ones.
        \item\label{A5} For every $f\in X^*$, there is a unique minimizer $x_f\in X_0$ of $F^f$ and the solution map
                \[ S\colon X_0^*\to X_0\quad\text{with } f\mapsto x^f \]
        is demi-continuous, i.e. maps strongly convergent sequences to weakly convergent ones.
    \end{enumerate}
\end{setting}
\begin{remark}
     As mentioned earlier, $(X,\norm{\cdot}_X)$ is usually a Sobolev space with its natural norm. The norm $\lvert\cdot\rvert$ may then chosen to be an $L^p(\Omega)$ or $W^{s,p}(\Omega)$ norm, where $s$ is strictly smaller than the differentiability order of $X$. In this case, Rellich's compactness theorem provides Assumption \ref{A4}.
\end{remark}

\begin{lemma}[Compactness]\label{SchauderCompactness}
    Assume we are in Setting \ref{setting:uniform_gamma_convergence}. Then the solution operator $S\colon (X,\lvert\cdot\rvert)^*\to (X_0,\lvert\cdot\rvert)$ is completely continuous, i.e., maps weakly convergent sequences to strongly convergent ones.
\end{lemma}
\begin{proof}
    We begin by clarifying what we mean with $S$ being defined on $(X,\lvert\cdot\rvert)^*$. Denote by $i$ the inclusion map $i\colon X_0\to X$ and consider
    \begin{align*}
        (X,\lvert\cdot\rvert)^*\xlongrightarrow{\operatorname{Id}^*}       (X,\norm{\cdot}_X)^* \xlongrightarrow{i^*} (X_0,\norm{\cdot}_X)^* \xlongrightarrow{S} (X_0, \norm{\cdot}_X) \xlongrightarrow{\operatorname{Id}}(X_0, \lvert \cdot\rvert).
    \end{align*}
    By abusing notation, always when we refer to $S$ as defined on $(X,\lvert\cdot\rvert)^*$ we mean the above composition, i.e., $\operatorname{Id}\circ S\circ i^*\circ\operatorname{Id}^*$. Having explained this, it is clear that it suffices to show that $\operatorname{Id}^*$ maps weakly convergent sequences to strongly convergent ones since \(i^*\) is continuous, \(S\) demi-continuous and \(\operatorname{Id}\) strongly continuous. This, however, is a consequence of Schauder's theorem, see for instance \cite{alt1992linear}, which states that a linear map $L\in\mathcal{L}(X,Y)$ between Banach spaces is compact if and only if $L^*\in\mathcal{L}(Y^*,X^*)$ is. Here, compact means that $L$ maps bounded sets to relatively compact ones. Let $X_c$ denote the completion of $(X,\lvert\cdot\rvert)$. Then, using the reflexivity of $(X,\norm{\cdot}_X)$ it is easily seen that $\operatorname{Id}\colon (X,\norm{\cdot}_X)\to X_c$ is compact. Finally, using that $(X,\lvert\cdot\rvert)^* = X_c^*$ the desired compactness of $\operatorname{Id}^*$ is established.
\end{proof}

The following theorem is the main result of this section. It shows that the convergence of the Deep Ritz method is uniform on bounded sets in the space $(X,\,\lvert\cdot\rvert\,)^*$. The proof of the uniformity follows an idea from \cite{cherednichenko2018norm}, where in a different setting a compactness result was used to amplify pointwise convergence to uniform convergence across bounded sets, compare to Theorem 4.1 and Corollary 4.2 in \cite{cherednichenko2018norm}.
\begin{theorem}[Uniform Convergence of the Deep Ritz Method]\label{absunires}
    Assume that we are in Setting \ref{setting:uniform_gamma_convergence} and let $\delta_n\searrow0$ be a sequence of real numbers. For $f\in X^*$ we set
    \[ S_n(f) \coloneqq \left\{ x\in X\;\big\lvert\; F^f_n(x) \leq \inf_{z\in X} F_n^f(z) + \delta_n\right\}, \]
    which is the approximate solution set corresponding to  $f$ and $\delta_n$. Furthermore, denote the unique minimizer of $F^f$ in $X_0$ by $x^f$ and fix $R>0$. Then we have
    \[ \sup \Big\{ \lvert x^f_n-x^f\rvert \;\big\lvert\; x_n^f\in S_n(f), \ \lVert f\rVert_{(X,\,\lvert\cdot\rvert\,)^*} \leq R \Big\} \to 0 \quad \text{for } n\to\infty. \]
\end{theorem}
In the definition of this supremum, $f$ is measured in the norm of the space $(X,\,\lvert\cdot\rvert\,)^*$. This means that $f:(X,\,\lvert\cdot\rvert\,)\to\mathbb{R}$ is continuous which is a more restrictive requirement than the continuity with respect to $\lVert \cdot \rVert_X$. Also the computation of this norm takes place in the unit ball of $(X,\,\lvert\cdot\rvert\,)$, i.e.
\[
    \lVert f \rVert_{(X,\,\lvert\cdot\rvert\,)^*} = \sup_{\lvert x \rvert \leq 1}f(x).
\]
Before we prove Theorem \ref{absunires} we need a $\Gamma$-convergence result similar to Theorem \ref{thm:abstract_gamma_convergence_thm}. The only difference is, that now also the right-hand side may vary along the sequence.

\begin{proposition}\label{theProposition}
    Assume that we are in Setting \ref{setting:uniform_gamma_convergence}, however, we do not need Assumption \ref{A5} for this result. Let $f_n,f\in(X,\lvert\cdot\rvert)^*$ such that $f_n\rightharpoonup f$ in the weak topology of the reflexive space $(X,\lvert\cdot\rvert)^*$. Then the sequence $(F_n^{f_n})_{n\in\mathbb{N}}$ of functionals $\Gamma$-converges to $F^f$ in the weak topology of $(X,\norm{\cdot}_X)$. Furthermore, the sequence $(F_n^{f_n})_{n\in\mathbb{N}}$ is equi-coercive.
\end{proposition}
\begin{proof}
    The proof is almost identical to the one of Theorem \ref{thm:abstract_gamma_convergence_thm} but since it is brief, we include it for the reader's convenience. We begin with the limes inferior inequality.
    Let $x_n\rightharpoonup x$ in $X$ and $x\notin X_0$. Then $x_n\to x$ with respect to $\lvert\cdot\rvert$ which implies that $f_n(x_n)$ converges to $f(x)$. Using that $ \gamma(x_n)\rightharpoonup\gamma(x) $ in $B$ combined with the weak lower semicontinuity of $\rVert\cdot\lVert^p_B $ we get
    \[
    \liminf_{n\to\infty}F_n^{f_n}(x_n) \geq \inf_{x\in X} E(x) + \liminf_{n\to\infty}\lambda_n\lVert\gamma(x_n)\rVert^p_B - \lim_{n\to\infty}f_n(x_n) = \infty.
    \]
    Now let $x\in X_0$. Then by the weak lower semicontinuity of $E$ we find
    \[
    \liminf_{n\to\infty}F_n^{f_n}(x_n) \geq \liminf_{n\to\infty}E(x_n) - f(x) \geq E(x) -f(x) = F^f(x).
    \]
    Now let us have a look at the construction of the recovery sequence. For $x\notin X_0$ we can choose the constant sequence and estimate
    \[
    F^{f_n}_n(x) \geq \inf_{x\in X}E(x) + \lambda_n\lVert\gamma(x)\rVert_B^p - \lVert f_n\rVert_{(X,\lvert\cdot\rvert)'} \cdot \lvert x\rvert.
    \]
    As $\lVert f_n\rVert_{(X,\lvert\cdot\rvert)^*}$ is bounded we find $F_n^{f_n}(x)\to \infty = F^f(x)$. If $x\in X_0$ we approximate it with a sequence $(x_n)\subseteq X$ according to Assumption \ref{A1}, such that $x_n\in A_n$ and $x_n\to x$ in $\lVert\cdot\rVert_X$ and $\lambda_n\lVert \gamma(x_n) \rVert_B^p \to 0$. It follows that
    \[
    F_n^{f_n}(x_n) = E(x_n) + \lambda_n\lVert x_n \rVert_B^p -f_n(x_n) \to E(x)-f(x) = F^f(x).
    \]
    The equi-coercivity was already assumed in \ref{A3} so it does not need to be shown.
\end{proof}
\begin{proof}[Proof of Theorem \ref{absunires}]
    We can choose $(f_n)\subseteq(X,\lvert\cdot\rvert)^*$ and $\lVert f_n\rVert_{(X,\lvert\cdot\rvert)^*}\leq R$ and $x_n^{f_n}\in S_n(f_n)$ such that
    \[
        \sup_{\begin{subarray}{c} \lVert f\rVert_{(X,\,\lvert\cdot\rvert\,)^*}\leq R \\ x^f_n\in S_n(f) \end{subarray}}\big\lvert x^f_n-x^f\big\rvert \leq \big\lvert x_n^{f_n}-x^{f_n}\big\rvert + \frac{1}{n}.
    \]
    Now it suffices to show that $\lvert x_n^{f_n} - x^{f_n}\rvert$ converges to zero. Since $(f_n)_{n\in\mathbb{N}}$ is bounded in $(X,\lvert\cdot\rvert)^*$ and this space is reflexive we can without loss of generality assume that $f_n\rightharpoonup f$ in $(X,\lvert\cdot\rvert)^*$. This implies by Lemma \ref{SchauderCompactness} that $x^{f_n}\to x^f$ in $(X,\lvert\cdot\rvert)$. The $\Gamma$-convergence result of the previous proposition yields $x_n^{f_n}\rightharpoonup x^f$ in $X$ and hence $x_n^{f_n}\to x^f$ with respect to $\lvert\cdot\rvert$ which concludes the proof.
\end{proof}

\section{Examples}\label{section:examples}
We discuss different concrete examples that allow the application of our abstract results and focus on non-linear problems. In particular, we consider a phase field model illustrating the basic $\Gamma$-convergence result of Section \ref{section:plain_gamma_convergence} and the $p$-Laplacian as an example for the uniform results of Section \ref{section:uniform_gamma_covergence}.
\subsection{A Phase Field Model}
Let $\varepsilon > 0$ be fixed, $\Omega \subset \mathbb{R}^d$ a bounded Lipschitz domain and consider the following energy
\begin{equation*}
    E\colon H^1(\Omega) \cap L^4(\Omega) \to [0,\infty), \quad E(u) = \frac{\varepsilon}{2}\int_\Omega |\nabla u|^2\mathrm dx + \frac{1}{\varepsilon}\int_\Omega W(u)\mathrm dx,
\end{equation*}
where $W\colon\mathbb{R}\to \mathbb{R}$ is a non-linear function, given by
\begin{equation*}
    W(u) = \frac14u^2(u-1)^2 = \frac14 u^4 - \frac12 u^3 + \frac14 u^2.
\end{equation*}
The functional $E$ constitutes a way to approximately describe phase separation and the parameter $\varepsilon$ encodes the length-scale of the phase transition, see \cite{cahn1958free}. We describe now how the Setting \ref{setting:gamma_convergence} is applicable to fully connected ReLU neural network ansatz functions. For the Banach spaces in Setting \ref{setting:gamma_convergence} we choose
\begin{gather*}
    X = H^1(\Omega)\cap L^4(\Omega), \quad B = L^2(\partial\Omega), \quad \norm{\cdot}_X = \norm{\cdot}_{H^1(\Omega)} + \norm{\cdot}_{L^4(\Omega)}, \quad \norm{\cdot}_B =\norm{\cdot}_{L^2(\partial\Omega)}.
\end{gather*}
These spaces are clearly reflexive and the trace operator meets the requirements of continuity and linearity and is our choice for $\gamma$, together with $p=2$. For the sets $(A_n)_{n\in\mathbb{N}}$ we use the ReLU activation function and define
\[
    A_n\coloneqq \left\{ u_\theta \mid \theta \in \Theta_n \right\}\subset H^1(\Omega)\cap L^4(\Omega),
\]
where $\Theta_n$ encodes that we use scalar valued neural networks with input dimension $d$ and depth $\lceil \log_2(d+1)\rceil +1$. The width of all other layers is set to $n$. Then it holds $A_n\subset A_{n+1}$ for all $n\in\mathbb{N}$ and Theorem~\ref{UniversalApproximation} shows that Assumption \ref{A1} is satisfied.

The continuity of $E$ with respect to $\norm{\cdot}_X$ is clear, hence we turn to the weak lower semi-continuity. To this end, we write $E$ in the following form
\begin{align*}
    E(u) = \underbrace{\frac{\varepsilon}{2}\int_\Omega |\nabla u|^2\mathrm dx + \frac{1}{4\varepsilon}\int_\Omega u^4\mathrm dx}_{\eqqcolon E_1(u)} + \underbrace{\frac{1}{\varepsilon} \int_\Omega \frac14 u^2 - \frac12 u^3\mathrm dx }_{\eqqcolon E_2(u)}
\end{align*}
and treat $E_1$ and $E_2$ separately. The term $E_1$ is continuous with respect to $\norm{\cdot}_X$ and convex, hence weakly lower semi-continuous. To treat $E_2$, note that we have the compact embedding
\begin{equation*}
    H^1(\Omega)\cap L^4(\Omega) \hookrightarrow\hookrightarrow L^3(\Omega).
\end{equation*}
This implies that a sequence that converges weakly in $H^1(\Omega)\cap L^4(\Omega)$ converges strongly in $L^3(\Omega)$ and consequently shows that the term $E_2$ is continuous with respect to weak convergence in $X$. Finally, for fixed $f\in X^*$, we need to show that the sequence $(F^f_n)_{n\in\mathbb{N}}$ is equi-coercive with respect to $\norm{\cdot}_X$. To this end, it suffices to show that the sequence
\begin{equation*}
    G^f_n\colon X\to\mathbb{R}, \quad G^f_n(u) = \frac{\varepsilon}{2}\int_\Omega |\nabla u|^2\mathrm dx + \frac{1}{\varepsilon}\int_\Omega W(u)\mathrm dx + \lambda_n\int_{\partial\Omega}u^2\mathrm ds - f(u)
\end{equation*}
is equi-coercive as it holds $F^f_n \geq G^f_n$. Let $r\in\mathbb{R}$ be fixed and consider all $u\in X$ with $G^f_n(u) \geq r$. Then, without losing generality, we may assume $\lambda_n \geq 1$ and estimate
\begin{align*}
    r \geq G_n^f(u)
    &\geq
    \frac{\varepsilon}{2}\int_\Omega |\nabla u|^2\mathrm dx + \int_{\partial\Omega}u^2\mathrm ds + \frac{1}{\varepsilon}\int_\Omega W(u)\mathrm dx - f(u)
    \\
    &\geq
    c\norm{u}^2_{H^1(\Omega)} - \lVert f \rVert_{X^*}\left(\norm{u}_{H^1(\Omega)} + \lVert u \rVert_{L^4(\Omega)}\right) + \frac{1}{4\varepsilon}\norm{u}^4_{L^4(\Omega)} - \frac{1}{3\varepsilon}\norm{u}^3_{L^3(\Omega)}
    \\
    &\geq
    c\norm{u}^2_{H^1(\Omega)} - \lVert f \rVert_{X^*}\norm{u}_{H^1(\Omega)} + \frac{1}{4\varepsilon}\norm{u}^4_{L^4(\Omega)} - \frac{|\Omega|^{1/4}}{3\varepsilon}\norm{u}^{3/4}_{L^4(\Omega)} - \lVert f \rVert_{X^*}\lVert u \rVert_{L^4(\Omega)},
\end{align*}
where we used Friedrich's inequality, see Proposition~\ref{poinclemma} and the estimate
\begin{equation*}
    \norm{u}_{L^3(\Omega)}^3 \leq |\Omega|^{1/4}\norm{u}_{L^4(\Omega)}^{3/4}
\end{equation*}
due to H\"older's inequality. This clearly implies a bound on the set
\begin{equation*}
    \bigcup_{n\in\mathbb{N}}\left\{ u\in H^1(\Omega)\cap L^4(\Omega) \mid G_n^f(u) \leq r \right\}
\end{equation*}
and hence $(F^f_n)_{n\in\mathbb N}$ is equi-coercive.

\begin{remark}[Stability under Compact Perturbations]
    With a similar -- even simpler -- approach we may also show that energies of the form
    \begin{equation*}
        \hat E(u) = E(u) + F(u)
    \end{equation*}
    fall in the Setting \ref{setting:gamma_convergence} provided $E$ does and $F$ is bounded from below and continuous with respect to weak convergence in $X$. Note also, that in the space dimension $d = 2$ this includes the above example, however, the slightly more involved proof presented here works independently of the space dimension $d$.
\end{remark}

\begin{figure}
    \centering
    \begin{subfigure}{0.415\linewidth}
    \includegraphics[width=\linewidth]{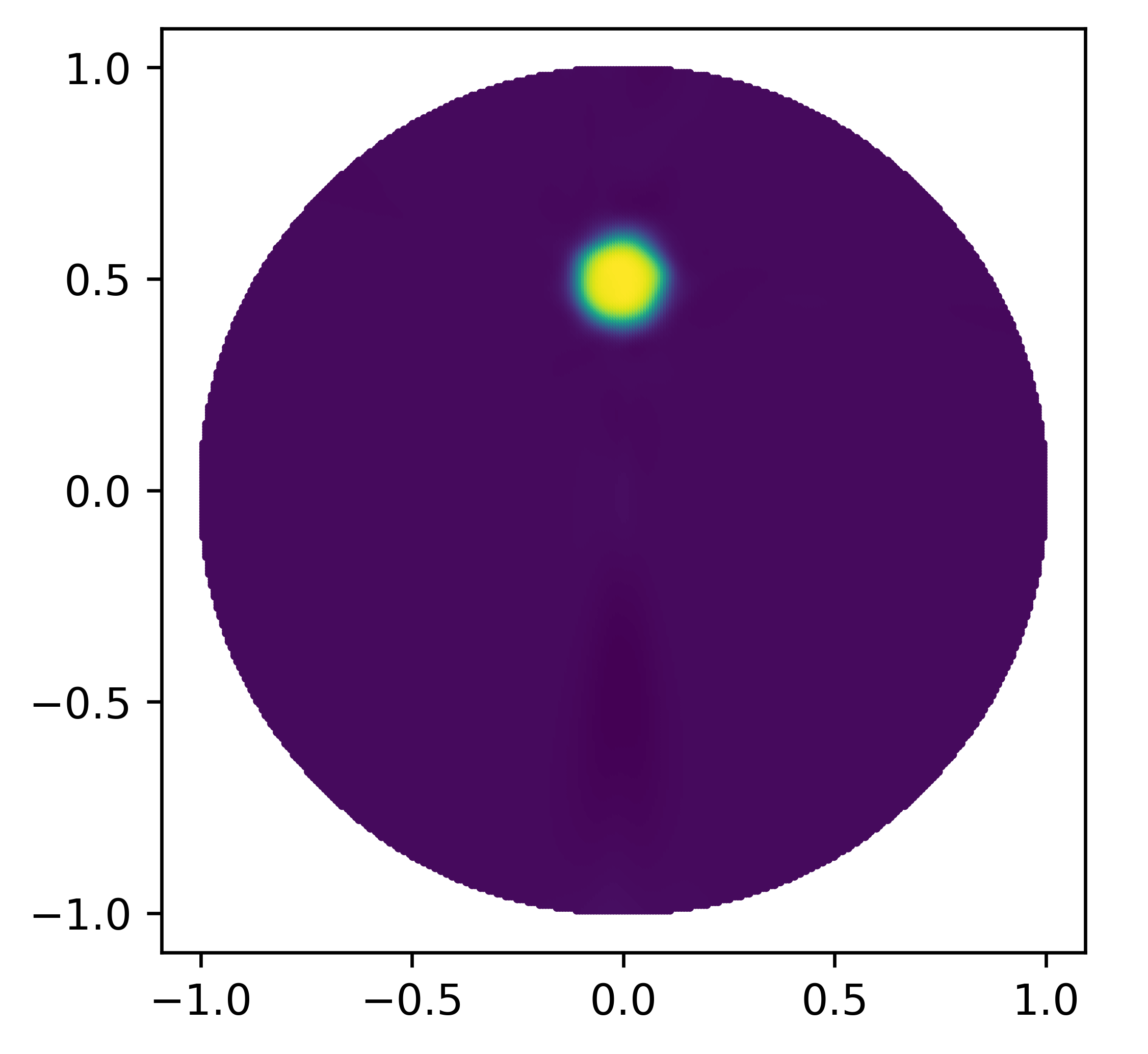}
    \end{subfigure}
    \begin{subfigure}{0.48\linewidth}
    \includegraphics[width=\linewidth]{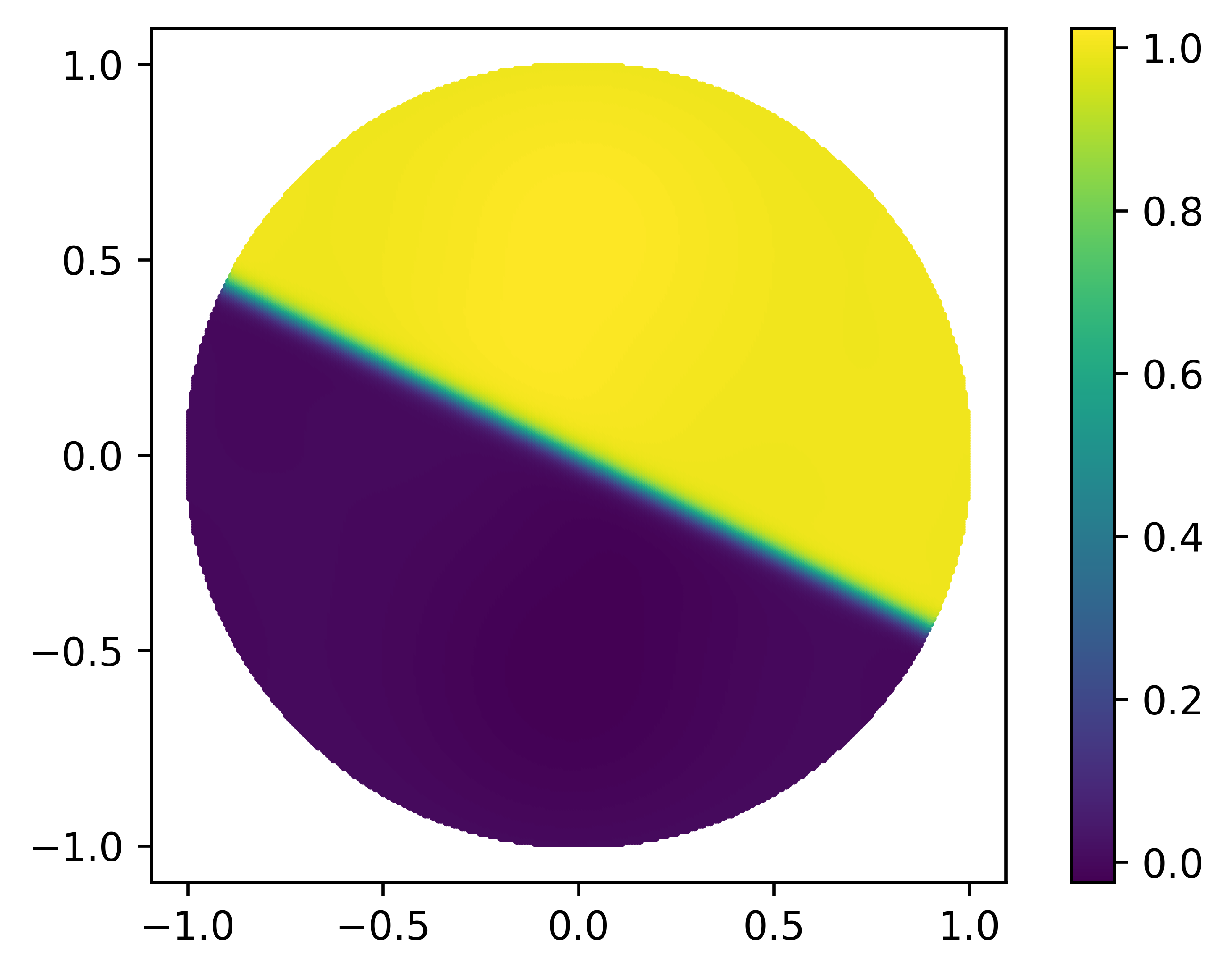}
    \end{subfigure}
    \caption{Exemplary numerical realization of the Deep Ritz Method for a Cahn-Hilliard functional with right-hand sides given through $f=\chi_{B_{r}(0,-1/2)} - \chi_{B_{r}(0,1/2)}$ with $r=0.1$ for the left plot and $r=0.4$ for the right plot. The value of $\varepsilon$ is set to $0.01$. We used zero Neumann boundary conditions and fully connected feed-forward networks with three hidden layers of width 16 and $\tanh$ activation. The number of trainable parameters is $609$.}
    \label{fig:modica_mortula}
\end{figure}

\begin{remark}
Figure~\ref{fig:modica_mortula} shows two exemplary numerical realizations of the Deep Ritz Method with right-hand sides
\begin{equation*}
    f_i = \chi_{B_{r_i}(0,-1/2)} - \chi_{B_{r_i}(0,1/2)}
\end{equation*}
for $r_1=0.1$ and $r_2=0.4$ corresponding to the left and right picture. Note that in the case of $f_1$, a phase transition around the ball $B_{r_1}(0,1/2)$ is energetically more favorable than the configuration in the right figure, where the radius $r_2$ is much larger.
\end{remark}

\subsection{\texorpdfstring{The $p$-Laplacian}{The p-Laplacian}}
\begin{figure}
    \centering
    \begin{subfigure}{0.48\linewidth}
    \includegraphics[width=\linewidth]{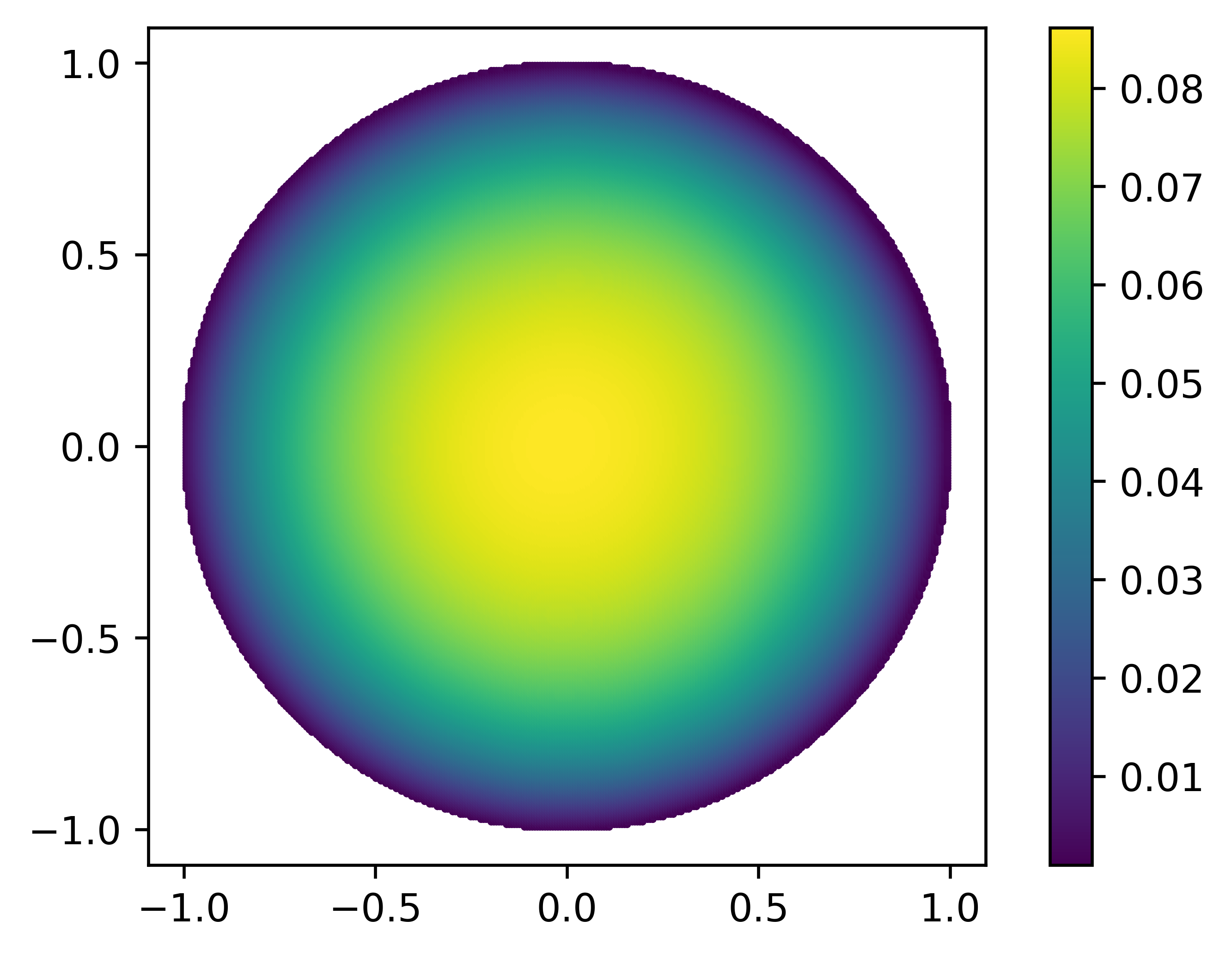}
    \end{subfigure}
    \begin{subfigure}{0.48\linewidth}
    \includegraphics[width=\linewidth]{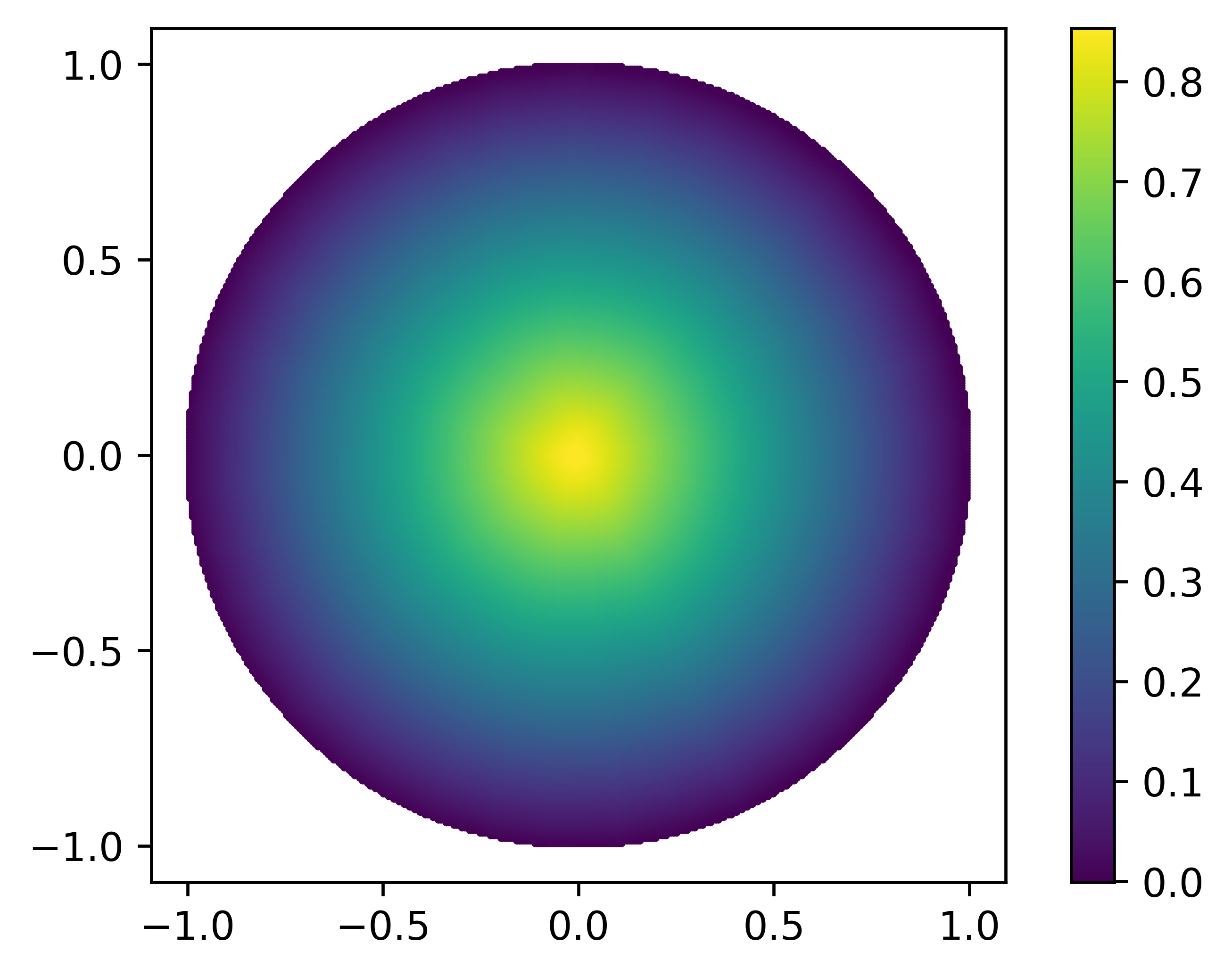}
    \end{subfigure}
    \caption{Exemplary numerical realization of the Deep Ritz Method for the $p$-Laplacian with right-hand $f=1$ and $p=1.5$ in the left plot and $p=10$ in the right plot. Zero Dirichlet boundary conditions are enforced through a penalty parameter $\lambda=250$. We used fully connected feed-forward networks with three hidden layers of width 16 and GELU activation for the left plot and ReLU activation for the right plot. The number of trainable parameters is $609$. Note the difference in the scaling of the axis in the two plots.}
    \label{fig:p_laplace}
\end{figure}
As an example for the uniform convergence of the Deep Ritz method we discuss the $p$-Laplacian. To this end, consider the $p$-Dirichlet energy for $p\in (1,\infty)$ given by
\begin{align*}
    E\colon W^{1,p}(\Omega)\to\mathbb{R}
    ,\quad
    u \mapsto
    \frac1p\int_\Omega|\nabla u|^p\,\mathrm dx.
\end{align*}
Note that for $p\neq 2$ the associated Euler-Lagrange equation – the $p$-Laplace equation –  is nonlinear. In strong formulation it is given by
\begin{align*}
    -\operatorname{div}(\lvert\nabla u\rvert^{p-2}\nabla u) &= f \quad\text{in }\Omega
    \\
    u &= 0 \quad\text{on }\partial\Omega,
\end{align*}
see for example \cite{struwe1990variational} or \cite{ruzicka2006nichtlineare}. Choosing the ReLU activation function, the abstract setting is applicable as we will describe now. For the Banach spaces we choose
\[ X = W^{1,p}(\Omega), \quad B = L^p(\partial\Omega), \quad \lvert u\rvert = \norm{u}_{L^p(\Omega)}\]
where the norms $\norm{\cdot}_X$ and $\norm{\cdot}_B$ are chosen to be the natural ones. Clearly, $W^{1,p}(\Omega)$ endowed with the norm $\norm{\cdot}_{W^{1,p}(\Omega)}$ is reflexive by our assumption $p\in(1,\infty)$. Note that it holds
    \[ \left(W^{1,p}(\Omega),\norm{\cdot}_{L^p(\Omega)}\right)^* = L^p(\Omega)^* \cong L^{p^\prime}(\Omega), \]
which is also reflexive.
We set $\gamma = \operatorname{tr}$, i.e.
\begin{align*}
    \operatorname{tr}\colon W^{1,p}(\Omega) & \to L^p(\partial\Omega)\quad\text{with}\quad u\mapsto u|_{\partial\Omega}
\end{align*}
We use the same ansatz sets $(A_n)_{n\in\mathbb{N}}$ as in the previous example, hence Assumption \ref{A1} holds.
Rellich's theorem provides the complete continuity of the embedding
    \[ \left( W^{1,p}(\Omega),\norm{\cdot}_{W^{1,p}(\Omega)}\right) \to \left(W^{1,p}(\Omega),\norm{\cdot}_{L^{p}(\Omega)}\right) \]
which shows Assumption \ref{A4}. As for Assumption \ref{A3}, Friedrich's inequality provides the assumptions of Lemma \ref{lemma:coercivity_condition}. Furthermore, $E$ is continuous with respect to $\norm{\cdot}_{W^{1,p}(\Omega)}$ and convex,
hence also weakly lower semi-continuous. By Poincar\'e's and Young's inequality we find for all $u\in W_0^{1,p}(\Omega)$ that
\begin{align*}
    F^f(u) &= \frac1p \int_\Omega \lvert\nabla u\rvert^p\mathrm dx - f(u)
    \\&\geq C\norm{u}^p_{W^{1,p}(\Omega)} - \lVert f\rVert_{W^{1,p}(\Omega)'}\norm{u}_{W^{1,p}(\Omega)}
    \\&\geq C\norm{u}^p_{W^{1,p}(\Omega)} - \tilde{C}.
\end{align*}
Hence, a minimizing sequence in $W^{1,p}_0(\Omega)$ for $F^f$ is bounded and as $F^f$ is strictly convex on $W^{1,p}_0(\Omega)$ it possesses a unique minimizer. Finally, to provide the demi-continuity we must consider the operator $S\colon W_0^{1,p}(\Omega)^*\to W_0^{1,p}(\Omega)$ mapping $f$ to the unique minimizer $u_f$ of $E - f$ on $W^{1,p}_0(\Omega)$. By the Euler-Lagrange formalism, \(u\) minimizes \(F^f\) if and only if
\[ \int_\Omega \lvert\nabla u\rvert^{p-2}\nabla u \cdot \nabla v \mathrm dx  = f(v) \quad\text{for all } v \in W_0^{1,p}(\Omega).\]
Hence, the solution map $S$ is precisely the inverse of the mapping
\[ W_0^{1,p}(\Omega)\to W^{1,p}_0(\Omega)^*, \quad u\mapsto \left( v\mapsto\int_\Omega \lvert\nabla u\rvert^{p-2}\nabla u \cdot \nabla v \mathrm dx \right) \]
and this map is demi-continuous, see for example \cite{ruzicka2006nichtlineare}.
\begin{remark}
    Figure~\ref{fig:p_laplace} shows two numerical realizations of the Deep Ritz Method for the $p$-Laplacian with right-hand side $f\equiv 1$ and $p_1=3/2$ in the left picture and $p_2=10$ in the right picture. The penalization value is set to $\lambda = 250$ in both simulations to approximately enforce zero boundary values. Note that the exact solution to the homogeneous $p$-Laplace problem on the disk with $f\equiv 1$ is given by
    \begin{equation*}
        u_p(x) = C\cdot\left( 1 - |x|^{\frac{p}{p-1}} \right)
    \end{equation*}
    for a suitable constant $C$ that depends on the spatial dimension and the value of $p$. We see that the solution $u_p$ converges pointwise to zero for $p\searrow 0$ and for $p\nearrow\infty$ the function $u_p$ tends to $x\mapsto C(1-|x|)$. This asymptotic behavior is clearly visible in our simulations.
\end{remark}

\bibliographystyle{apalike}
\bibliography{references}

\appendix

\section{Universal approximation with zero boundary values}

Here we prove the universal approximation result which we stated as Theorem \ref{UniversalApproximation} in the main text. Our proof uses that every continuous, piecewise linear function can be represented by a neural network with ReLU activation function and then shows how to approximate Sobolev functions with zero boundary conditions by such functions. The precise definition of a piecwise linear function is the following.

\begin{definition}[Continuous piecewise linear function]\label{PiecewiseLinear}
We say a function \(f \colon\mathbb R^d \to\mathbb R\) is  \emph{continuous piecewise linear} or shorter \emph{piecewise linear} if there exists a finite set of closed polyhedra whose union is \(\mathbb R^d\), and \(f\) is affine linear over each polyhedron. Note every piecewise linear functions is continuous by definition since the polyhedra are closed and cover the whole space \(\mathbb R^d\), and affine functions are continuous.
\end{definition}
\begin{theorem}[Universal expression]\label{ArorasTheorem}
    Every ReLU neural network function $u_\theta:\mathbb{R}^d\to\mathbb{R}$
    is a piecewise linear function. Conversely, every piecewise linear function \(f\colon\mathbb R^d\to\mathbb R\) can be expressed by a ReLU network of depth at most \(\lceil \log_2(d+1)\rceil +1\).
\end{theorem}
For the proof of this statement we refer to \cite{arora2016understanding}.
We turn now to the approximation capabilities of piecewise linear functions.
\begin{lemma}\label{ApproxLemma}
    Let $\varphi\in C_c^\infty(\mathbb{R}^d)$ be a smooth function with compact support. Then for every $\varepsilon>0$ there is a piecewise linear function $s_\varepsilon$ such that for all $p\in[1,\infty]$ it holds
    \begin{align*}
        \lVert s_\varepsilon-\varphi\rVert_{W^{1,p}(\mathbb R^d)}\le\varepsilon
        \quad \text{and}\quad
        \operatorname{supp}(s_\varepsilon)
        \subseteq
        \operatorname{supp}(\varphi) + B_{\varepsilon}(0).
    \end{align*}
    Here, we set $B_{\varepsilon}(0)$ to be the $\varepsilon$-ball around zero, i.e. $B_{\varepsilon}(0)=\{ x\in\mathbb{R}\mid \lvert x\rvert < \varepsilon \}$.
    \begin{proof}
        In the following we will denote by $\norm{\cdot}_\infty$ the uniform norm on $\mathbb R^d$. To show the assertion choose a triangulation $\mathcal{T}$ of $\mathbb{R}^d$ of width $\delta=\delta(\varepsilon) >0$, consisting of rotations and translations of one non-degenerate simplex $K$. We choose $s_\varepsilon$ to agree with $\varphi$ on all vertices of elements in $\mathcal{T}$. Since $\varphi$ is compactly supported it is uniformly continuous and hence it is clear that $\lVert \varphi-s_\varepsilon\rVert_\infty <\varepsilon$ if $\delta$ is chosen small enough.

        To show convergence of the gradients we show that also $\lVert \nabla \varphi - \nabla s_\varepsilon\rVert_{\infty}<\varepsilon$ which will be shown on one element $K\in\mathcal{T}$ and as the estimate is independent of $K$ is understood to hold on all of $\mathbb{R}^d$.
        So let $K\in\mathcal{T}$ be given and denote its vertices by $x_1,\dots,x_{d+1}$. We set $v_i = x_{i+1}-x_1$, $i=1,\dots,d$ to be the vectors spanning $K$. By the one dimensional mean value theorem we find $\xi_i$ on the line segment joining $x_1$ and $x_i$ such that
        \[
            \partial_{v_i}s_\varepsilon(v_1) = \partial_{v_i}\varphi(\xi_i).
        \]
        Note that $\partial_{v_i}s_\varepsilon$ is constant on all of $K$ where it is defined. Now for arbitrary $x\in K$ we compute with setting $w=\sum_{i=1}^d\alpha_iv_i$ for $w\in\mathbb{R}^d$ with $|w|\leq 1$. Note that the $\alpha_i$ are bounded uniformly in $w$, where we use that all elements are the same up to rotations and translations.
        \begin{align*}
            \lvert \nabla\varphi(x) - \nabla s_\varepsilon(x)\rvert
            &=
            \sup_{\lvert w\rvert\leq1}\lvert \nabla\varphi(x)w-\nabla s_\varepsilon(x)w\rvert
            \\
            &\leq
            \sup_{\lvert w\rvert\leq 1} \sum_{i=1}^d \lvert\alpha_i\rvert\cdot
            \underbrace{
            \lvert \partial_{v_i}\varphi(x) - \partial_{v_i}s_\varepsilon(x) \rvert
            }_{=(*)}
        \end{align*}
        where again $(*)$ is uniformly small due to the uniform continuity of $\nabla\varphi$. Noting that the $W^{1,\infty}$-case implies the claim for all $p\in[1,\infty)$ finishes the proof.
    \end{proof}
\end{lemma}
We turn to the proof of Theorem \ref{UniversalApproximation} which we state again for the convenience of the reader.

\begin{theorem}[Universal approximation with zero boundary values]
Consider an open set \(\Omega\subseteq\mathbb R^d\) and let \(u\in W^{1,p}_0(\Omega)\) with $p\in[1,\infty)$. Then for all \(\varepsilon>0\) there exists a function \(u_\varepsilon\in W^{1,p}_0(\Omega)\) that can be expressed by a ReLU network of depth \(\lceil \log_2(d+1)\rceil +1\) such that
    \[ \left\lVert u - u_\varepsilon \right\rVert_{W^{1,p}(\Omega)}\le \varepsilon. \]
\begin{proof}
Let $u\in W^{1,p}_0(\Omega)$ and $\varepsilon>0$. By the density of $C_c^\infty(\Omega)$ in $W^{1,p}_0(\Omega)$, see for instance \cite{brezis2010functional}, we choose a smooth function $\varphi_\varepsilon\in C^\infty_c(\Omega)$ such that $\lVert u-\varphi_\varepsilon\rVert_{W^{1,p}(\Omega)}\leq \varepsilon/2$. Furthermore we use Lemma \ref{ApproxLemma} and choose a piecewise linear function $u_\varepsilon$ such that $\lVert \varphi_\varepsilon-u_\varepsilon\rVert_{W^{1,p}(\Omega)}\leq\varepsilon/2$ and such that $u_\varepsilon$ has compact support in $\Omega$.
This yields
\[\lVert u-u_\varepsilon\rVert_{W^{1,p}(\Omega)}\leq\lVert u-\varphi_\varepsilon\rVert_{W^{1,p}(\Omega)}+\lVert \varphi_\varepsilon-u_\varepsilon\rVert_{W^{1,p}(\Omega)}\leq\varepsilon\]
and by Theorem \ref{ArorasTheorem} we know that $u_\varepsilon$ is in fact a realisation of a neural network with depth at most \(\lceil \log_2(d+1)\rceil +1\).
\end{proof}
\end{theorem}

\end{document}